\documentclass[10pt,twoside]{siamart1116}

\usepackage[english]{babel}
\usepackage{graphicx,epstopdf}
\usepackage{amsfonts,epsfig,fancyhdr,amsmath,amssymb}

\newcommand{\Names}{Hiroki Minamide}
\newcommand{\Title}{On the Adjacency spectra of alternating-oriented \texorpdfstring{\lowercase{$n$}}{n}-gonal staircase digraphs
}

\newcommand{\Spec}{\mathrm{Spec}}
\newcommand{\ord}{\mathrm{ord}}
 
\renewcommand{\theequation}{\thesection.\arabic{equation}}
\renewtheorem{theorem}{Theorem}[section]
\renewtheorem{definition}[theorem]{Definition}

\newsiamremark{remark}{Remark}
\newsiamremark{example}{Example}

\makeatletter
\def\cref@override@label@type#1\@nil#2{#1}
\makeatother

\begin{document}

\bibliographystyle{plain}

\setcounter{page}{1}

\thispagestyle{empty}

   \title{\Title
 }

\author{Hiroki Minamide
\thanks{Department of Liberal Arts, National Institute of Technology, Tokyo College, Hachioji, Tokyo, 193-0997, Japan.\newline
Department of Mathematical and Computing Science, School of Computing, Institute of Science Tokyo, Ookayama, Meguro-ku, Tokyo 152-8552 Japan.
\newline(minamide@tokyo-ct.ac.jp).
}
}

\markboth{\Names}{\Title}

\maketitle

\begin{abstract}
For integers $n\ge3$ and $r\ge1$, let $\Gamma_{n,r}$ be the alternating-oriented digraph formed by gluing $r$ directed $n$-cycles along a single edge in a staircase pattern, and
let $A_{n,r}$ be its adjacency matrix.
A canonical $n$-layer partition puts $A_{n,r}$ into an $n$-cyclic block form and isolates a cyclic product core $K_{n,r}$,
so the nonzero spectrum of $A_{n,r}$ is obtained from that of $K_{n,r}$ by an $n$th-root lift.
We show that $K_{n,r}$ is totally nonnegative and irreducible, hence its nonzero eigenvalues are real, positive, and simple.
It follows that all nonzero eigenvalues of $A_{n,r}$ are simple and occur in $\exp(2\pi i/n)$-orbits, forming unions of regular $n$-gons in the complex plane.
A one-step Schur complement produces a three-term recursion in $r$ for the characteristic polynomials $\Phi_{n,r}\in\mathbb Z[x]$ of $A_{n,r}$.
This determines both the zero multiplicity and the nonzero spectral count and leads to a cubic-denominator generating function.
Applying Tran-type confinement gives the uniform bound for the spectral radius $\rho(A_{n,r})\le (27/4)^{1/n}$, which is sharp in the sense that
$\displaystyle\lim_{r\to\infty}\rho(A_{n,r})=(27/4)^{1/n}$ for each fixed $n$.
Finally, specializing at $x=1$ links $\Phi_{n,r}(1)$ to Padovan’s spiral numbers and yields a complete classification of rational nonzero eigenvalues.
\end{abstract}




\section{Introduction}

For integers $n\ge3$ and $r\ge1$, let $\Gamma_{n,r}$ be the alternating-oriented digraph obtained by gluing $r$ directed $n$-cycles in a staircase pattern, and let $A_{n,r}$ be its adjacency matrix.
Figure~\ref{fig:Gamma_examples} illustrates the examples $\Gamma_{3,8}$ and $\Gamma_{4,6}$.
Our goal is to describe the adjacency spectrum of $A_{n,r}$ in the complex plane.

\paragraph{Related work}
Spectral problems for digraph families obtained by gluing directed $n$-cycles along shared edges have recently become a fertile testing ground for polygonal spectral patterns.
In such constructions, the gluing parameter $r$ counts the number of polygonal blocks and governs the size of the resulting digraph.
The line of work originated from chemically motivated models: for alternating-oriented linearly connected hexagons (modeling \emph{polyacenes}, i.e., linear chains of fused benzene rings), Kashihara--Asano--Minamide proved that the spectrum decomposes into regular hexagons via explicit recurrences for the characteristic polynomial \cite{KAM2023}.
This viewpoint was then extended to other gluing schemes.
For alternating-oriented ladder graphs, Ichikawa--Kashihara--Minamide showed that the characteristic polynomial is related to that of an undirected path graph, and hence to Fibonacci polynomials \cite{IKM2024}.
For alternating-oriented comb graphs, Kashihara--Minamide identified a connection between the characteristic polynomial and that of an undirected comb graph, and hence with Pell polynomials \cite{KM2025}.
For alternating-oriented annular graphs, a similar connection with the characteristic polynomial of an undirected cycle graph, and hence with Chebyshev polynomials of the first kind, was identified in \cite{KM}.

A common feature of these works is that the characteristic polynomial can be identified with, or reduced to, a classical special-polynomial family, so that known root properties yield explicit spectral information.
Our staircase digraphs $\Gamma_{n,r}$ fit into this circle of ideas but behave differently.
The natural companion objects are the $k$-Narayana numbers viewed as polynomials in $k$, for which explicit root formulas are not available in general.
We therefore take a linear-algebraic route based on cyclic reduction and positivity.

\paragraph{Approach and main results}
A canonical $n$-layer partition puts $A_{n,r}$ into an $n$-cyclic (block-cyclic) normal form and isolates a cyclic product core $K_{n,r}$, so the nonzero spectrum of $A_{n,r}$ is obtained from that of $K_{n,r}$ by taking $n$th roots.
We prove that $K_{n,r}$ is totally nonnegative and irreducible, hence its nonzero eigenvalues are real, positive, and simple.
Consequently, the nonzero eigenvalues of $A_{n,r}$ are simple, and the nonzero spectrum forms $\exp(2\pi i/n)$-packets, i.e., unions of regular $n$-gons centered at the origin, each with a distinguished vertex on the positive real axis.
A one-step Schur complement yields a three-term recursion in the gluing parameter $r$ for the characteristic polynomials $\Phi_{n,r}\in\mathbb{Z}[x]$.
From this recursion we determine the algebraic multiplicity of the zero eigenvalue and the nonzero spectral count.
We then apply Tran’s root-confinement theorem \cite{Tran2014} to the associated cubic-denominator generating function, obtaining the uniform bound
$\rho(A_{n,r})\le (27/4)^{1/n}$ and the sharp limit $\displaystyle\lim_{r\to\infty}\rho(A_{n,r})=(27/4)^{1/n}$ for each fixed $n$.
Finally, specializing the recursion at $x=1$ connects $\Phi_{n,r}(1)$ to Padovan's spiral numbers and yields a complete classification of rational nonzero eigenvalues.
Figure~\ref{fig:spectrum_Gamma38} illustrates the polygonal spectrum for $\Gamma_{3,8}$.

\paragraph{Organization}
Sections~\ref{sec:cyclic-block-form}--\ref{sec:TN-simplicity} develop the $n$-cyclic reduction and the core theory, Section~\ref{sec:recursion-and-counts} derives the $r$-recursion and spectral counts,
Sections~\ref{sec:Tran}--\ref{sec:reconstruction} treat confinement, spectral geometry, and reconstruction questions, and Section~\ref{sec:rational} classifies rational nonzero eigenvalues;
the appendices collect comparisons and worked examples.

\section{Cyclic block form and the polygonal spectrum}\label{sec:cyclic-block-form}
The introduction pointed out a periodic feature of $\Gamma_{n,r}$ of length $n$.
In this section we turn that periodicity into a concrete normal form of the adjacency matrix.
Rather than fixing an explicit grading map $V(\Gamma_{n,r})\to\mathbb Z/n\mathbb Z$,
we construct a canonical partition of $V(\Gamma_{n,r})$ into $n$ layers along the recursion.
A single permutation that lists vertices layer-by-layer then reveals an $n$-cyclic block form,
from which the polygonal pattern of the nonzero eigenvalues follows by a standard cyclic-product reduction.

\subsection{The digraph $\Gamma_{n,r}$}\label{subsec:Gamma}
We now give a precise definition of the alternating-oriented $n$-gonal staircase digraph with $r$ blocks, denoted by $\Gamma_{n,r}$.
Let $n\ge 3$ and $r\ge 1$, and set
\[
  V(\Gamma_{n,r})=\{1,2,\dots, rn-2r+2\}.
\]
The directed edge set $E(\Gamma_{n,r})\subset V(\Gamma_{n,r})\times V(\Gamma_{n,r})$ is defined recursively by
\[
E(\Gamma_{n,1})=\{(j,j+1):1\le j\le n-1\}\cup\{(n,1)\},
\]
\[
E(\Gamma_{n,2})=E(\Gamma_{n,1})\cup\{(1,n+1)\}\cup\{(j,j+1):n+1\le j\le 2n-3\}\cup\{(2n-2,n)\},
\]
and, for $r\ge 3$,
\[
E(\Gamma_{n,r})=E(\Gamma_{n,2})\cup \bigcup_{i=1}^{r-2} F_i,
\]
where for each $i$ with $1\le i\le r-2$ we set
\begin{align*}
F_i
:=\{(in-2i+2,\; in+n-2i+1)\}
 &\cup\ \{(j,j+1): in+n-2i+1\le j\le in+2n-2i-3\}\\
 &\cup\ \{(in+2n-2i-2,\; in+n-2i)\}.
\end{align*}

\begin{remark}
When $n=3$, the middle path part in the definitions of $E(\Gamma_{n,2})$ and $F_i$ is empty,
so $E(\Gamma_{3,2})\setminus E(\Gamma_{3,1})$ and each $F_i$ consist of two edges.
The construction remains valid verbatim.
\end{remark}

See Figure~1 for the examples $\Gamma_{3,8}$ and $\Gamma_{4,6}$ appearing in the definition of $\Gamma_{n,r}$.

\begin{figure}[t]
\centering
\includegraphics[scale=.35]{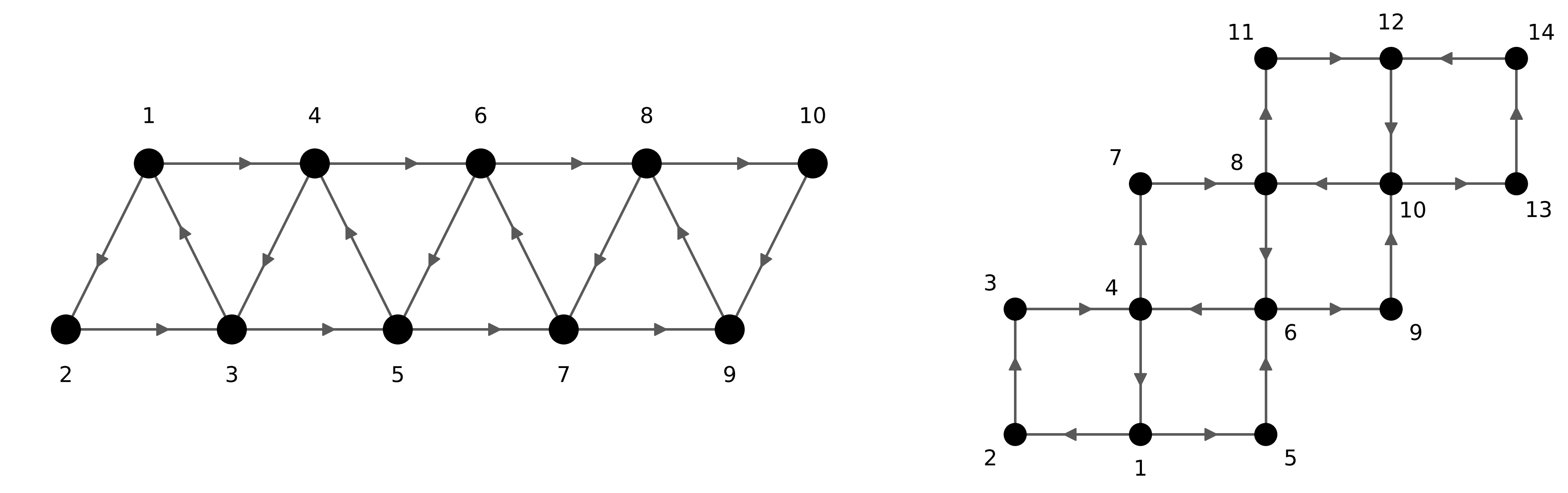}
\caption{Visualizations of the digraphs $\Gamma_{3,8}$ and $\Gamma_{4,6}$.}
\label{fig:Gamma_examples}
\end{figure}

We write $A_{n,r}$ for the adjacency matrix of $\Gamma_{n,r}$, i.e., the $(0,1)$-matrix indexed by $V(\Gamma_{n,r})$
whose $(u,v)$-entry is $1$ if and only if $(u,v)\in E(\Gamma_{n,r})$.
We set $\Phi_{n,r}:=\det(xI-A_{n,r})\in \mathbb{Z}[x]$ for the characteristic polynomial of $A_{n,r}$.
The digraph $\Gamma_{n,r}$ has an intrinsic period-$n$ structure: every edge belongs to one of the $n$-cycles created by the construction.
In the next subsection we formalize this by a canonical $n$-layer partition, which puts $A_{n,r}$ into an $n$-cyclic block form.

\subsection{Layer partition and the induced permutation}
We assign each vertex a layer index in $\mathbb Z/n\mathbb Z$ so that every directed edge advances the index by $1$ modulo $n$.

\begin{definition}[Layer partition]\label{def:layers}
Along the recursive construction of $E(\Gamma_{n,r})$, we define a map $\ell:V(\Gamma_{n,r})\to \mathbb Z/n\mathbb Z$ as follows.
Set $\ell(1)=0$.
Whenever a vertex $v$ appears for the first time as the terminal vertex of a newly added edge $(u,v)$, define
\[
\ell(v)\equiv \ell(u)+1 \pmod n.
\]
For each $c\in\mathbb Z/n\mathbb Z$, set
\[
V^{(c)}_{n,r}:=\{v\in V(\Gamma_{n,r}) : \ell(v)\equiv c\pmod n\}.
\]
\end{definition}

\noindent\emph{Well-definedness.}
Every vertex $v\neq 1$ appears for the first time as the terminal vertex of a newly added edge $(u,v)$ whose initial vertex $u$ already lies in the previously constructed digraph.
Hence $\ell(v)$ is assigned at least once, and the ``first appearance'' clause ensures it is assigned at most once.

\begin{lemma}[Layer increment along edges]\label{lem:layer-compat}
For every edge $(u,v)\in E(\Gamma_{n,r})$ we have $\ell(v)\equiv \ell(u)+1\pmod n$.
Equivalently, there are no edges within a layer, and every edge goes from $V_{n,r}^{(c)}$ to $V_{n,r}^{(c+1)}$ (indices modulo $n$).
\end{lemma}

\begin{proof}
By Definition~\ref{def:layers}, whenever a vertex $v$ appears for the first time as the terminal vertex of a newly added edge $(u,v)$, we assign
$\ell(v)\equiv \ell(u)+1\pmod n$.
Thus the congruence holds for every edge at the moment it is introduced unless its terminal vertex has already been assigned a layer.

The only edges whose terminal vertices are already assigned are the closing (exit) edges that attach a newly created path back to the previously constructed part.
Fix such an exit edge $(u,v)$.
The vertex $u$ lies on the newly created directed path, so $\ell(u)$ is obtained from the entrance edge by repeatedly adding $1$ along the path.
The vertex $v$ lies in the older part and already has an assigned layer.
Since the path from $v$ to $u$ has length $n-1$ and each of its edges advances the layer by $1$, we have
$\ell(u)\equiv \ell(v)+n-1\pmod n$, hence $\ell(v)\equiv \ell(u)+1\pmod n$.

Therefore $\ell(v)\equiv \ell(u)+1\pmod n$ for every edge $(u,v)\in E(\Gamma_{n,r})$, and the layer statement follows.
\end{proof}

\subsection{An $n$-cyclic block form}

Let $\sigma_{n,r}$ be the permutation that lists the vertices layer-by-layer,
first all vertices of $V_{n,r}^{(0)}$ in increasing order, then those of $V_{n,r}^{(1)}$, and so on up to $V_{n,r}^{(n-1)}$.
Let $S_{n,r}$ be the corresponding permutation matrix, and define
\[
  B_{n,r}:=S_{n,r} A_{n,r} S_{n,r}^{-1}.
\]
From now on we work with $B_{n,r}$ and its blocks. Since $B_{n,r}$ is permutation similar to $A_{n,r}$, they have the same spectrum, and we will not switch back to $A_{n,r}$ unless explicitly stated.

\begin{proposition}[$n$-cyclic block form]\label{prop:ncyclic-form}
With respect to the layer-by-layer ordering
\[
V(\Gamma_{n,r}) = V_{n,r}^{(0)} \sqcup V_{n,r}^{(1)} \sqcup \cdots \sqcup V_{n,r}^{(n-1)},
\]
the matrix $B_{n,r}$ has the block form
\[
B_{n,r}=
\begin{bmatrix}
0 & B_{n,r}^{(0)} & 0 & \cdots & 0 & 0\\
0 & 0 & B_{n,r}^{(1)} & \cdots & 0 & 0\\
\vdots & \vdots & \ddots & \ddots & \vdots & \vdots\\
0 & 0 & \cdots & 0 & B_{n,r}^{(n-3)} & 0\\
0 & 0 & \cdots & 0 & 0 & B_{n,r}^{(n-2)}\\
B_{n,r}^{(n-1)} & 0 & \cdots & 0 & 0 & 0
\end{bmatrix},
\]
where $B_{n,r}^{(c)}$ encodes the edges from $V_{n,r}^{(c)}$ to $V_{n,r}^{(c+1)}$
(indices modulo $n$).
\end{proposition}

\begin{proof}
By Lemma~\ref{lem:layer-compat}, every edge of $\Gamma_{n,r}$ goes from $V_{n,r}^{(c)}$ to $V_{n,r}^{(c+1)}$
(indices modulo $n$).
Hence, in the layer-by-layer order $\sigma_{n,r}$, the adjacency matrix $B_{n,r}$ has nonzero blocks only in positions $(c,c+1)$, yielding the displayed form.
\end{proof}

\subsection{Cyclic product reduction and polygonal eigenvalues}\label{subsec:reduction}

The $n$-cyclic block form of $B_{n,r}$ makes the role of the layers transparent.
A single application of $B_{n,r}$ sends each layer to the next one, and after $n$ steps one returns to the original layer.
This yields an $n$-step return map on each layer, represented by a square matrix.
The point of introducing this return map is that, once its eigenvalues are known, the nonzero eigenvalues of $A_{n,r}$ are obtained by taking $n$th roots.
In particular, the nonzero eigenvalues occur in $\exp(2\pi i/n)$-orbits and hence form regular $n$-gons in the complex plane.

With the $n$-cyclic block form in Proposition~\ref{prop:ncyclic-form}, each block $B_{n,r}^{(c)}$ records the edges from
$V_{n,r}^{(c)}$ to $V_{n,r}^{(c+1)}$ (indices modulo $n$).
We define the cyclic product core as the $n$-step return map on a single layer,
\[
K_{n,r}:=B_{n,r}^{(0)}B_{n,r}^{(1)}\cdots B_{n,r}^{(n-1)}.
\]

The following proposition describes the reduction to the core $K_{n,r}$ at the level of nonzero eigenvalues and characteristic polynomials.
\begin{proposition}[Cyclic product reduction]\label{prop:charpoly-reduction}
The nonzero eigenvalues of $A_{n,r}$ are obtained by taking $n$th roots of the nonzero eigenvalues of $K_{n,r}$, with algebraic multiplicities preserved.
Equivalently, there exists an integer $\nu_{n,r}\ge0$ such that
\[
\Phi_{n,r}(x)=x^{\nu_{n,r}}\det(x^nI-K_{n,r}).
\]
\end{proposition}

\begin{proof}
Since \(A_{n,r}\) is permutation similar to \(B_{n,r}\), it suffices to work with \(B_{n,r}\).
By Proposition~\ref{prop:ncyclic-form}, the matrix \(B_{n,r}^{\,n}\) is block diagonal, and its diagonal blocks are the cyclic products
\[
K_{n,r}^{(c)}:=B_{n,r}^{(c)}B_{n,r}^{(c+1)}\cdots B_{n,r}^{(c-1)},
\]
where the superscripts are read modulo \(n\).
These blocks have the same nonzero spectrum, with the same algebraic multiplicities, since cyclic products \(XY\) and \(YX\) have the same nonzero spectrum; see, e.g., \cite[Theorem~1.3.22]{HornJohnson2013}.
In particular, each \(K_{n,r}^{(c)}\) has the same nonzero spectrum as \(K_{n,r}\).
Therefore the nonzero eigenvalues of \(B_{n,r}^{\,n}\) are exactly those of \(K_{n,r}\), each repeated \(n\) times.
Hence the nonzero eigenvalues of \(B_{n,r}\) are exactly the \(n\)th roots of the nonzero eigenvalues of \(K_{n,r}\), with algebraic multiplicities preserved.
This is equivalent to the displayed factorization of \(\Phi_{n,r}\).
\end{proof}

This reduction isolates the nonzero spectrum in the core $K_{n,r}$, while the factor $x^{\nu_{n,r}}$ records only part of the contribution of the zero eigenvalue.
The next remark explains this distinction.

\begin{remark}[Zero multiplicity vs.\ the reduction exponent]\label{rem:nu-vs-ord}
We write $\ord_x\Phi$ for the order of vanishing of $\Phi\in\mathbb Z[x]$ at $x=0$.
From Proposition~\ref{prop:charpoly-reduction},
\[
\Phi_{n,r}=x^{\nu_{n,r}}\det(x^nI-K_{n,r}),
\qquad\text{hence}\qquad
\ord_x\Phi_{n,r}=\nu_{n,r}+\ord_x\det(x^nI-K_{n,r}).
\]
The second term records the contribution of the zero eigenvalue of $K_{n,r}$, and therefore comes in multiples of $n$.
In particular, $\nu_{n,r}$ is a \emph{reduction exponent} rather than the full zero multiplicity; the exact value of
$\ord_x\Phi_{n,r}$ will be pinned down later by the $r$-recurrence.
\end{remark}

\section{Total nonnegativity and simplicity of the nonzero spectrum}\label{sec:TN-simplicity}

By the cyclic reduction in Section~\ref{subsec:reduction}, it suffices to control the core $K_{n,r}$; the nonzero spectrum of
$A_{n,r}$ is then obtained by an $n$th-root lift.
We prove that $K_{n,r}$ is totally nonnegative and irreducible; hence the nonzero eigenvalues of $K_{n,r}$ are positive and simple.
Proposition~\ref{prop:charpoly-reduction} then transfers this simplicity to $A_{n,r}$.

\smallskip
We keep the notation from Proposition~\ref{prop:ncyclic-form}.

\subsection{The block factors are totally nonnegative}\label{subsec:blocks-TN}

The $n$-cyclic normal form does more than reorganize indices:
it turns the combinatorics of each gluing step into a monotone incidence pattern.
This is precisely the entry point of total nonnegativity.

\begin{proposition}[Total nonnegativity of the core]\label{prop:K-TN}
The cyclic product $K_{n,r}$ is totally nonnegative, i.e., every square subdeterminant of $K_{n,r}$ is nonnegative.
\end{proposition}

\begin{proof}
Fix $c$. With vertices listed in increasing order within each layer, the block $B^{(c)}_{n,r}$ records edges from
$V^{(c)}_{n,r}$ to $V^{(c+1)}_{n,r}$.
We claim that, after all gluing steps, the support of $B^{(c)}_{n,r}$ is contained in at most two adjacent diagonals.

Indeed, in each gluing step, the new vertices form a directed path whose incidences respect the increasing order in both layers,
so the new path contributes a single diagonal string of ones.
The closing edge attaches at an endpoint of the newly appended segment, hence it can create at most one additional incidence,
lying on a diagonal adjacent to the one coming from the path.
Because the recursive construction always places newly created vertices after the previously existing ones in each layer order, every gluing step only enlarges the current bidiagonal support at its boundary.
Therefore $B^{(c)}_{n,r}$ is supported on at most two adjacent diagonals and is thus bidiagonal.

Every bidiagonal $0$--$1$ matrix is totally nonnegative: every square submatrix is either zero or triangular with nonnegative diagonal, hence has nonnegative determinant.
Since total nonnegativity is preserved under products by the Cauchy--Binet formula; see, e.g., \cite[Theorem~1.1.2]{FallatJohnson2011},
the cyclic product $K_{n,r}=B^{(0)}_{n,r}\cdots B^{(n-1)}_{n,r}$ is totally nonnegative.
\end{proof}

We next prove that $K_{n,r}$ is irreducible, so the simplicity theorem applies.

\subsection{Irreducibility}\label{subsec:irr-inv}

A square matrix is called irreducible if it cannot be transformed into block upper triangular form by a simultaneous permutation of rows and columns.
A digraph is called strongly connected if, for every ordered pair of vertices, there exists a directed walk from the first to the second.
The link between these two notions is the key point in what follows.

\begin{lemma}[Strong connectivity of $\Gamma_{n,r}$]\label{lem:Gamma-SC}
The digraph $\Gamma_{n,r}$ is strongly connected.
\end{lemma}

\begin{proof}
View $\Gamma_{n,r}$ as a chain of directed $n$-cycles $C_1,\dots,C_r$, where $C_{k+1}$ is glued to $C_k$ along a shared directed edge $(a_k,b_k)$ for $k=1,\dots,r-1$.
Fix $u\in V(C_i)$ and $v\in V(C_j)$.
Since each $C_k$ is strongly connected, there is a directed walk inside $C_i$ from $u$ to $a_i$.
If $i<j$, traverse $(a_i,b_i)$ to enter $C_{i+1}$ at $b_i$, and iterate along the chain until reaching $C_j$; then move inside $C_j$ to $v$.
(The case $i>j$ is analogous.)
Concatenating these walks gives a directed walk $W$ in $\Gamma_{n,r}$ from $u$ to $v$, so $\Gamma_{n,r}$ is strongly connected.
\end{proof}

\begin{proposition}[Irreducibility of the core]\label{prop:K-irreducible}
The matrix $K_{n,r}$ is irreducible.
\end{proposition}

\begin{proof}
Let \(G_{n,r}\) be the digraph whose vertex set is \(V^{(0)}_{n,r}\), and in which there is a directed edge \((u,v)\) precisely when the \((u,v)\)-entry of \(K_{n,r}\) is positive.

Let \(u,v\in V^{(0)}_{n,r}\).
By Lemma~\ref{lem:Gamma-SC}, there is a directed walk \(W\) in \(\Gamma_{n,r}\) from \(u\) to \(v\).
By Lemma~\ref{lem:layer-compat}, each edge advances the layer by \(1\) modulo \(n\). Hence the length of \(W\) is a multiple of \(n\).
Partition \(W\) into subwalks of length \(n\).
Each such subwalk starts and ends in \(V^{(0)}_{n,r}\), and therefore determines a length-\(n\) directed walk through the cyclic sequence of layers.
Equivalently, the corresponding \((u,v)\)-entry of
$K_{n,r}=B^{(0)}_{n,r}\cdots B^{(n-1)}_{n,r}$
is positive, since this entry counts length-\(n\) directed walks from \(u\) to \(v\) that pass once through the cyclic sequence of layers.
Thus \(G_{n,r}\) contains a directed walk from \(u\) to \(v\).
Since \(u\) and \(v\) were arbitrary, \(G_{n,r}\) is strongly connected.
Hence \(K_{n,r}\) is irreducible by \cite[Theorem~6.2.24]{HornJohnson2013}.
\end{proof}

\subsection{Simple nonzero spectrum}\label{subsec:osc-simple}

Total nonnegativity controls signs but not separation.
Irreducibility is the missing spark: it forces the positive spectrum to split into simple eigenvalues.

\begin{theorem}[Positivity and simplicity of the nonzero spectrum]\label{thm:K-pos-simple}
Every nonzero eigenvalue of $K_{n,r}$ is real, positive, and simple.
\end{theorem}

\begin{proof}
By Proposition~\ref{prop:K-TN}, the matrix $K_{n,r}$ is totally nonnegative, and by
Proposition~\ref{prop:K-irreducible} it is irreducible. Hence, by \cite[Theorem~5.5.18(1)]{FallatJohnson2011},
all nonzero eigenvalues of $K_{n,r}$ are positive and distinct.
\end{proof}

By Proposition~\ref{prop:charpoly-reduction}, the nonzero eigenvalues of $A_{n,r}$ are obtained by taking $n$th roots
of the eigenvalues of $K_{n,r}$, counted with algebraic multiplicity.
Since each nonzero eigenvalue of $K_{n,r}$ is simple by Theorem~\ref{thm:K-pos-simple}, the same holds for the nonzero eigenvalues of $A_{n,r}$.

\begin{corollary}[Simplicity of the nonzero spectrum]\label{cor:simple-nonzero}
Every nonzero eigenvalue of $A_{n,r}$ is simple.
\end{corollary}

\section{A recursion for $\Phi_{n,r}$ and zero/nonzero spectral counts}\label{sec:recursion-and-counts}

The cyclic reduction in Sections~\ref{sec:cyclic-block-form}--\ref{sec:TN-simplicity} turns spectral questions for $A_{n,r}$ into questions about the characteristic polynomials
\[
\Phi_{n,r}(x)=\det(xI-A_{n,r}).
\]
In this section we derive a three-term recursion in the gluing parameter $r$ from a single local determinant computation.
The recursion serves as our bookkeeping tool: it determines the multiplicity of the zero eigenvalue, namely $\ord_x\Phi_{n,r}$, and, via
\[
\deg\Phi_{n,r}=|V(\Gamma_{n,r})|=rn-2r+2,
\]
also the number of nonzero eigenvalues.
We keep in mind that $\nu_{n,r}$ from Proposition~\ref{prop:charpoly-reduction} is a reduction exponent,
whereas the zero multiplicity is $\ord_x\Phi_{n,r}$.

\subsection{A three-term recursion in the gluing parameter}\label{subsec:Phi-rec}

One gluing step changes the graph only locally, but its characteristic polynomial remembers the change globally.

\begin{lemma}[A three-term recursion]\label{lem:Phi-rec}
For $r\ge4$, the characteristic polynomials satisfy
\[
\Phi_{n,r}=x^{\,n-2}\Phi_{n,r-1}-x^{\,2(n-3)}\Phi_{n,r-3}.
\]
\end{lemma}

\begin{proof}
Write \(xI-A_{n,r}\) in block form by separating the \(n-2\) vertices added in the last gluing step from the previously constructed part:
\[
xI-A_{n,r}=
\begin{bmatrix}
M_{n,r-1} & U\\
V & T_n
\end{bmatrix},
\]
where \(M_{n,r-1}=xI-A_{n,r-1}\), the block \(T_n\) corresponds to the new \((n-2)\)-vertex path, and \(U,V\) encode the two attachment edges.
Since \(\det(T_n)=x^{n-2}\), Schur complementation gives
\[
\Phi_{n,r}=x^{n-2}\det(M_{n,r-1}-UT_n^{-1}V).
\]
The matrix \(UT_n^{-1}V\) has a single nonzero entry, located at the position corresponding to the two attachment vertices, and this entry equals \((T_n^{-1})_{1,n-2}=x^{-(n-2)}\).
Hence \(M_{n,r-1}-UT_n^{-1}V\) is obtained from \(M_{n,r-1}\) by modifying a single entry, so
\[
\det(M_{n,r-1}-UT_n^{-1}V)=\det(M_{n,r-1})-x^{-(n-2)}\Delta,
\]
where \(\Delta\) is the corresponding cofactor.
Removing the row and column defining \(\Delta\) separates the terminal part of the staircase from the previously constructed subgraph \(\Gamma_{n,r-3}\).
With the remaining vertices ordered so that those of \(\Gamma_{n,r-3}\) come first, the cofactor matrix becomes block upper triangular.
Its leading diagonal block is exactly \(M_{n,r-3}\), and the remaining two diagonal blocks correspond to the two directed paths created by the last two gluing steps.
Each of these path blocks is upper triangular of size \(n-3\), with diagonal entries all equal to \(x\), and therefore has determinant \(x^{n-3}\).
Hence
\[
\Delta=x^{2(n-3)}\det(M_{n,r-3}).
\]
Substituting this identity into the Schur complement expansion above yields the claimed recurrence.
\end{proof}

\subsection{Zero multiplicity and the nonzero spectral count}\label{subsec:zero-and-nonzero}

Once the recurrence is in hand, the power of $x$ dividing $\Phi_{n,r}$ becomes an explicit periodic function of $r$.

\begin{proposition}[Zero multiplicity]\label{prop:zero-mult}
The algebraic multiplicity of the eigenvalue $0$ of $A_{n,r}$ equals

\[
\ord_x\Phi_{n,r}=rn-2r+2-n\Bigl\lfloor\frac{r+2}{3}\Bigr\rfloor
=\begin{cases}
\dfrac{1}{3}(2n-6)(r-1), & r\equiv 1\pmod 3,\\[3mm]
\dfrac{1}{3}(2n-6)(r-2)+(n-2), & r\equiv 2\pmod 3,\\[3mm]
\dfrac{1}{3}(2n-6)(r-3)+2(n-2), & r\equiv 0\pmod 3.
\end{cases}
\]
\end{proposition}

\begin{proof}
Let
\[
m_{n,r}:=rn-2r+2-n\Bigl\lfloor\frac{r+2}{3}\Bigr\rfloor .
\]
We first prove that \(x^{m_{n,r}}\) divides \(\Phi_{n,r}\), and then show that the coefficient of \(x^{m_{n,r}}\) does not vanish.

A direct check shows that
\[
m_{n,r}=
\begin{cases}
2(n-3)+m_{n,r-3}, & r\equiv1\pmod3,\\[2mm]
(n-2)+m_{n,r-1}=2(n-3)+m_{n,r-3}, & r\equiv0,2\pmod3.
\end{cases}
\]
Hence Lemma~\ref{lem:Phi-rec} implies inductively that \(m_{n,r}\le \ord_x\Phi_{n,r}\) for all \(r\ge1\).

Let \(c_{n,r}\) denote the coefficient of \(x^{m_{n,r}}\) in \(\Phi_{n,r}\). Then
\[
c_{n,r}=
\begin{cases}
-\,c_{n,r-3}, & r\equiv1\pmod3,\\[2mm]
c_{n,r-1}-c_{n,r-3}, & r\equiv0,2\pmod3.
\end{cases}
\]
Since
\[
c_{n,1}=-1,\qquad c_{n,2}=-2,\qquad c_{n,3}=1,
\]
the first relation fixes the signs of \(c_{n,3q+1}\).
It then follows inductively from the second relation that \(c_{n,3q+2}=c_{n,3q+1}-c_{n,3q-1}\) and \(c_{n,3q+3}=c_{n,3q+2}-c_{n,3q}\) are differences of opposite-sign terms, and hence are nonzero.
Therefore \(c_{n,r}\neq0\) for every \(r\ge1\), so \(\ord_x\Phi_{n,r}=m_{n,r}\).

The final piecewise form is an immediate consequence of the floor expression by considering the residue class of \(r\) modulo \(3\).
\end{proof}

The nonzero spectral count is now immediate from $\deg\Phi_{n,r}=rn-2r+2$ and Proposition~\ref{prop:zero-mult}.

\begin{corollary}[Nonzero spectral count]\label{cor:nonzero-count}
The number of nonzero eigenvalues of $A_{n,r}$ is
\[
\deg\Phi_{n,r}-\ord_x\Phi_{n,r}
=
n\Bigl\lfloor\frac{r+2}{3}\Bigr\rfloor.
\]
\end{corollary}

The recursion in Lemma~\ref{lem:Phi-rec} will be used in two ways: it fixes the spectral count in this section,
and it also yields a cubic-denominator generating function.
In the next section we apply a Tran-type confinement theorem to this generating function to obtain a uniform spectral bound.

\section{Tran-type confinement and the spectral radius}\label{sec:Tran}

In this section we convert the $r$-recursion from Section~\ref{sec:recursion-and-counts} into a generating-function identity and invoke a confinement theorem of Tran.
This yields a uniform disk bound for the spectrum and, combined with simplicity, explains the polygonal geometry in the complex plane.

\subsection{A Tran-type root confinement}\label{subsec:Tran-confinement}
The cubic denominator of the generating function places us exactly in the setting of Tran's theorem.

\begin{lemma}\label{lem:Phi-H}
Define polynomials $H_m(z)\in\mathbb{Z}[z]$ by
\[
\sum_{m\ge0}H_m(z)t^m=\frac{1}{1-t+zt^3}.
\]
Then, for all $r\ge1$,
\[
\Phi_{n,r}(x)=x^{(n-2)r+2}\,H_{r+2}(x^{-n}).
\]
\end{lemma}

\begin{proof}
From the generating function we have $H_0=H_1=H_2=1$ and the recursion
$H_m=H_{m-1}-zH_{m-3}$ for $m\ge3$.
Writing $z=x^{-n}$ and multiplying by $x^{(n-2)r+2}$ gives, for $r\ge4$,
\[
x^{(n-2)r+2}H_{r+2}(x^{-n})
= x^{n-2}\,x^{(n-2)(r-1)+2}H_{r+1}(x^{-n})
 - x^{2(n-3)}\,x^{(n-2)(r-3)+2}H_{r-1}(x^{-n}).
\]
This matches the recursion in Lemma~\ref{lem:Phi-rec}.
Since the displayed identity agrees with $\Phi_{n,r}$ for $r=1,2,3$ by direct computation,
the result follows for all $r\ge1$ by induction.
\end{proof}

\begin{lemma}\label{lem:no-boundary-root}
For every $m\ge0$, one has
\[
H_m\!\left(\frac{4}{27}\right)>0.
\]
In particular, no nonzero root $\lambda$ of $\Phi_{n,r}$ satisfies $\lambda^n=27/4$.
\end{lemma}

\begin{proof}
At $z=4/27$,
\[
\sum_{m\ge0}H_m\!\left(\frac{4}{27}\right)t^m
=
\frac{1}{1-t+\frac{4}{27}t^3}
=
\frac{1}{\left(1-\frac23 t\right)^2\left(1+\frac13 t\right)}.
\]
A partial-fraction decomposition gives
\[
\frac{1}{\left(1-\frac23 t\right)^2\left(1+\frac13 t\right)}
=
\frac{1}{9}\frac{1}{1+\frac13 t}
+\frac{2}{9}\frac{1}{1-\frac23 t}
+\frac{2}{3}\frac{1}{\left(1-\frac23 t\right)^2}.
\]
Expanding each term as a power series,
\[
\frac{1}{1+\frac13 t}=\sum_{m\ge0}\left(-\frac13\right)^m t^m,\qquad
\frac{1}{1-\frac23 t}=\sum_{m\ge0}\left(\frac23\right)^m t^m,\qquad
\frac{1}{\left(1-\frac23 t\right)^2}=\sum_{m\ge0}(m+1)\left(\frac23\right)^m t^m.
\]
Comparing coefficients of $t^m$ gives
\[
H_m\!\left(\frac{4}{27}\right)
=
\frac19\left(-\frac13\right)^m
+\frac29\left(\frac23\right)^m
+\frac23(m+1)\left(\frac23\right)^m
=
\frac{(-1)^m+(6m+8)2^m}{9\cdot 3^m}>0.
\]
Now let $\lambda\neq0$ be a root of $\Phi_{n,r}$.
If $\lambda^n=27/4$, then Lemma~\ref{lem:Phi-H} gives
\[
0=\Phi_{n,r}(\lambda)=\lambda^{(n-2)r+2}H_{r+2}(\lambda^{-n})
=\lambda^{(n-2)r+2}H_{r+2}\!\left(\frac{4}{27}\right),
\]
which is impossible because $\lambda\neq0$ and $H_{r+2}(4/27)>0$.
\end{proof}

\begin{theorem}[Tran-type confinement]\label{thm:Tran}
Fix $n\ge 3$ and $r\ge 1$. Every nonzero root $\lambda$ of $\Phi_{n,r}$ satisfies
\[
\mathrm{Im}(\lambda^n)=0,
\qquad
0<\mathrm{Re}(\lambda^n)< \frac{27}{4}.
\]
\end{theorem}

\begin{proof}
Let $\lambda\ne0$ be a root of $\Phi_{n,r}$.
By Lemma~\ref{lem:Phi-H} we have $H_{r+2}(\lambda^{-n})=0$.
Apply Tran's confinement theorem \cite[Theorem~3]{Tran2014} to
\[
\sum_{m\ge0}H_m(z)t^m=\frac{1}{1+Bt+At^3}
\quad\text{with}\quad B=-1,\ \ A=z.
\]
At $z=\lambda^{-n}$, Tran's parameter is
\[
-\frac{B^3}{A}=\frac{1}{z}=\lambda^n.
\]
Hence $\mathrm{Im}(\lambda^n)=0$ and $0\le \mathrm{Re}(\lambda^n)\le 27/4$.
Since $\lambda\ne0$, we have $\lambda^n\ne0$, so in fact $\mathrm{Re}(\lambda^n)>0$.
Moreover, Lemma~\ref{lem:no-boundary-root} rules out the boundary value $\lambda^n=27/4$.
Therefore
\[
0<\mathrm{Re}(\lambda^n)<\frac{27}{4}.
\]
\end{proof}

\subsection{A two-parameter limit for the Perron root}\label{subsec:Perron-limit}
We use the density statement in Tran's confinement theorem to upgrade the uniform bound in Theorem~\ref{thm:Tran} to a limit as $r\to\infty$.

\begin{proposition}[Sharpness of the Tran bound]\label{prop:sharp-Tran}
Fix $n\ge3$. Then
\[
\lim_{r\to\infty}\rho(A_{n,r})=\Bigl(\frac{27}{4}\Bigr)^{\frac{1}{n}}.
\]
\end{proposition}

\begin{proof}
Since $\Gamma_{n,r-1}$ is an induced subgraph of $\Gamma_{n,r}$, the matrix $A_{n,r-1}$ is a principal submatrix of $A_{n,r}$.
For nonnegative matrices, the spectral radius is monotone under taking principal submatrices; see, e.g.,
\cite[Corollary~8.1.20(a)]{HornJohnson2013}. Hence $\rho(A_{n,r})$ is nondecreasing in $r$.
By Theorem~\ref{thm:Tran}, $\rho(A_{n,r})\le (27/4)^{1/n}$ for all $r$, so $\displaystyle\lim_{r\to\infty}\rho(A_{n,r})$ exists.

It remains to show $\displaystyle\limsup_{r\to\infty}\rho(A_{n,r})\ge (27/4)^{1/n}$.
In the specialization $B=-1$, $A=z$ from the proof of Theorem~\ref{thm:Tran},
Tran's theorem \cite[Theorem~3]{Tran2014} implies that the set of parameters $1/z$ arising from roots $z$ of $H_m$ with
$\mathrm{Im}(1/z)=0$ is dense in $(0,27/4]$ as $m\to\infty$.

Let $\varepsilon>0$.
Choose arbitrarily large $m$ and a root $z_m$ of $H_m$ with $\mathrm{Im}(1/z_m)=0$ and $1/z_m>27/4-\varepsilon$.
Write $m=r+2$ and choose the positive real number $\lambda$ such that $\lambda^n=1/z_m$.
By Lemma~\ref{lem:Phi-H}, this $\lambda$ is a nonzero root of $\Phi_{n,r}$, hence an eigenvalue of $A_{n,r}$.
Therefore
\[
\rho(A_{n,r})\ge \lambda=\left(\frac{1}{z_m}\right)^{\frac{1}{n}}>\left(\frac{27}{4}-\varepsilon\right)^{\frac{1}{n}}.
\]
Since this holds for arbitrarily large \(r\), and since \(\varepsilon>0\) was arbitrary, we obtain
\[
\limsup_{r\to\infty}\rho(A_{n,r})\ge \left(\frac{27}{4}\right)^{\frac{1}{n}}.
\]
Since we also have the uniform upper bound $\rho(A_{n,r})\le (27/4)^{1/n}$, the proof is complete.
\end{proof}

\subsection{Regular $n$-gons in the complex plane}\label{subsec:ngons}

Now the algebra turns into geometry: every packet becomes a rigid regular $n$-gon.

\begin{theorem}[Polygonal spectrum]\label{thm:complex-spectrum}
Let $n\ge 3$ and $r\ge 1$.
The spectrum of $A_{n,r}$ consists of the eigenvalue $0$ with algebraic multiplicity $\ord_x\Phi_{n,r}$,
together with exactly $\lfloor (r+2)/3\rfloor$ regular $n$-gons,
all contained in the open disk of radius $(27/4)^{1/n}$, each $n$-gon having a vertex on the positive real axis.
\end{theorem}

\begin{proof}
Proposition~\ref{prop:zero-mult} gives the algebraic multiplicity of the eigenvalue $0$, and
Corollary~\ref{cor:nonzero-count} gives the number of nonzero eigenvalue packets.
By Proposition~\ref{prop:charpoly-reduction}, the nonzero eigenvalues of $A_{n,r}$ are obtained by taking $n$th roots of the nonzero eigenvalues of $K_{n,r}$, so they form regular $n$-gons.
Theorem~\ref{thm:Tran} gives the open-disk bound, and the positivity of the nonzero eigenvalues of $K_{n,r}$ implies that each such $n$-gon has a vertex on the positive real axis.
\end{proof}

Figure~\ref{fig:spectrum_Gamma38} illustrates Theorem~\ref{thm:complex-spectrum} for $\Gamma_{3,8}$.

\begin{figure}[t]
\centering
\includegraphics[scale=.4]{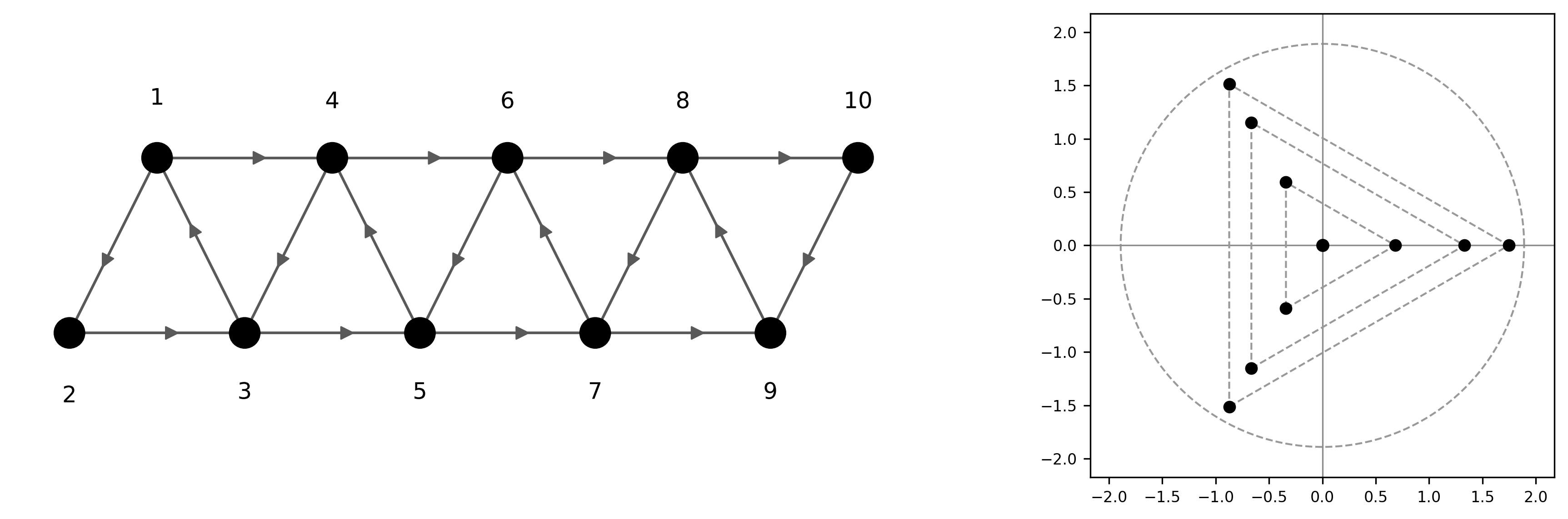}
\caption{The spectrum of $\Gamma_{3,8}$ and the dashed circle of radius $(27/4)^{1/3}$ in the complex plane.}
\label{fig:spectrum_Gamma38}
\end{figure}

\begin{remark}[Parity]\label{rem:parity}
A spectral $n$-gon has a \emph{vertex} on the negative real axis if and only if $n$ is even.
\end{remark}

\section{Reconstruction from the characteristic polynomial}\label{sec:reconstruction}
The three-term recursion from Section~\ref{sec:recursion-and-counts} leaves a rigid fingerprint in the top-degree part of $\Phi_{n,r}$.
In this section we show that this fingerprint determines the parameters $(n,r)$ from $\Phi_{n,r}$ alone.

\begin{lemma}[Binomial expansion for $\Phi_{n,r}$]\label{lem:Phi-binomial}
For every $n\ge 3$ and $r\ge 1$,
\[
\Phi_{n,r}(x)=
\sum_{j=0}^{\lfloor (r+2)/3\rfloor}
(-1)^j\binom{r+2-2j}{j}\,
x^{(n-2)r+2-nj}.
\]
\end{lemma}

\begin{proof}
For $n\ge3$ and $r\ge1$, let $\Psi_{n,r}\in\mathbb Z[x]$ denote the right-hand side.
Using the binomial identity
\[
\binom{m}{j}=\binom{m-1}{j}+\binom{m-1}{j-1},
\]
one checks that the family $\{\Psi_{n,r}\}_{r\ge1}$ satisfies the same three-term recursion as in Lemma~\ref{lem:Phi-rec}.
Moreover, $\Psi_{n,r}=\Phi_{n,r}$ for $r=1,2,3$ by direct verification.
Hence $\Psi_{n,r}=\Phi_{n,r}$ for all $r\ge1$.
\end{proof}

\begin{theorem}[Reconstruction from $\Phi_{n,r}$]\label{thm:reconstruct-Phi}
Suppose that a polynomial $\Phi\in\mathbb Z[x]$ is of the form $\Phi=\Phi_{n,r}$ for some $n\ge3$ and $r\ge1$.
Then the parameters $(n,r)$ are determined by $\Phi$.
\end{theorem}

\begin{proof}
By Lemma~\ref{lem:Phi-binomial}, the two highest-degree terms of $\Phi$ are
$x^{(n-2)r+2}-r\,x^{(n-2)r+2-n}$.
Hence $n$ is the gap between the top two exponents, and $r$ is the absolute value of the coefficient of the second-highest term.
\end{proof}

This reconstruction also explains why the arithmetic problem in Section~\ref{sec:rational} can be phrased purely in terms of $\Phi_{n,r}(1)$.

\section{Rational nonzero eigenvalues}\label{sec:rational}

In this section we determine when the spectrum of $A_{n,r}$ contains a rational nonzero eigenvalue.
The polygonal symmetry from Theorem~\ref{thm:complex-spectrum} suggests that such an eigenvalue, if it exists,
must occur at a distinguished vertex, and we will see that the problem reduces to the single specialization $\Phi_{n,r}(1)$.

Let $\lambda\in\mathbb{Q}$ be a nonzero eigenvalue of $A_{n,r}$.
Since $\Phi_{n,r}$ is monic, every eigenvalue is an algebraic integer, hence $\lambda$ must be an integer.
By Theorem~\ref{thm:Tran} we have $|\lambda|< (27/4)^{1/n}\le (27/4)^{1/3}<2$.
Thus the only possible rational nonzero eigenvalues are $\lambda=\pm1$.
Moreover, the polygonal structure implies that $-1$ can occur only when $n$ is even and $1$ belongs to the same spectral $n$-gon.
Hence it suffices to decide when $1$ is an eigenvalue of $A_{n,r}$, equivalently when $\Phi_{n,r}(1)=0$.

\subsection{Padovan's spiral numbers and de Weger's vanishing theorem}\label{sec:spiralPadovan}

We use the Padovan sequence $\{P_m\}_{m\in\mathbb Z}$ normalized by
\[
P_0=P_1=P_2=1,\qquad P_{m+3}=P_{m+1}+P_m\quad(m\in\mathbb Z),
\]
extended to all integers by the same recursion.
This normalization is commonly called \emph{Padovan's spiral numbers} and is listed
as \cite[A134816]{OEIS-A134816}.

The only external input we need is the classification of the zeros of $\{P_m\}$.
The following vanishing criterion goes back to de Weger; for a convenient explicit statement, see \cite[Table~1]{BravoBravoLuca2022}.

\begin{theorem}[de Weger]
\label{thm:deweger-zeros}
For the Padovan sequence $\{P_m\}_{m\in\mathbb Z}$ above,
\[
P_m=0 \quad\Longleftrightarrow\quad m\in\{-1,-3,-4,-8,-17\}.
\]
\end{theorem}

\subsection{From $\Phi_{n,r}(1)$ to Padovan}\label{subsec:Phi1-padovan}

The bridge from our characteristic polynomials to Padovan arithmetic is a specialization at $x=1$.
The recurrence in the gluing parameter $r$ collapses to the Padovan recursion, and the initial values match.

\begin{lemma}[Padovan values at $x=1$]\label{lem:padovan-values}
For every $n\ge3$ and $r\ge1$,
\[
\Phi_{n,r}(1)=(-1)^r\,P_{-r-7}.
\]
\end{lemma}

\begin{proof}
Evaluating the recursion in Lemma~\ref{lem:Phi-rec} at $x=1$ gives a three-term recursion in $r$ for $\Phi_{n,r}(1)$.
The same recursion holds for $(-1)^rP_{-r-7}$.
Since the two sequences agree for $r=1,2,3$ by direct verification, they agree for all $r$ by induction.
\end{proof}

We write $\Spec(M)$ for the spectrum of a matrix $M$, and $\mathbb{Q}^{\times}$ for the set of nonzero rational numbers.

\begin{theorem}[Padovan criterion]\label{thm:padovan-main}
Let $n\ge3$ and $r\ge1$. Then
\[
\Spec(A_{n,r})\cap\mathbb Q^\times\neq\emptyset
\quad\Longleftrightarrow\quad
1\in\Spec(A_{n,r})
\quad\Longleftrightarrow\quad
r\in\{1,10\}.
\]
\end{theorem}

\begin{proof}
As explained at the beginning of the section, to decide whether $\Spec(A_{n,r})\cap\mathbb{Q}^\times$ is nonempty,
it suffices to determine when $1$ is an eigenvalue of $A_{n,r}$, i.e., when $\Phi_{n,r}(1)=0$.

By Lemma~\ref{lem:padovan-values} we have $\Phi_{n,r}(1)=(-1)^r P_{-r-7}$, hence $\Phi_{n,r}(1)=0$ if and only if $P_{-r-7}=0$.
Now de Weger’s classification (Theorem~\ref{thm:deweger-zeros}) implies $r\in\{1,10\}$.
\end{proof}

This also explains the exceptional case $r=10$ illustrated below.

\begin{example}[The exceptional value $r=10$]\label{ex:r10}
At $r=10$ the rational eigenvalues are visible in the characteristic polynomial.
For instance, when $(n,r)=(4,10)$ one finds
\[
\Phi_{4,10}(x)=x^{6}(x-1)(x+1)(x^2+1)\bigl(x^{12}-9x^8+19x^4-1\bigr),
\]
so $x^4-1$ divides $\Phi_{4,10}(x)$ and $\pm1$ indeed occur as eigenvalues.
\end{example}

\appendix
\section{An arithmetic viewpoint}\label{app:kNarayana}
The $n=3$ characteristic factors satisfy a Narayana-type recursion in the gluing parameter.
This appendix records the corresponding $k$-Narayana numbers and packages the few arithmetic facts we use later.

\subsection{$k$-Narayana numbers: definition and generating function}
We use the $k$-Narayana numbers introduced in \cite{RamirezSirvent2015}.

\begin{definition}[$k$-Narayana polynomials {\cite[Section~1]{RamirezSirvent2015}}]\label{def:kNarayana}
Define $\{N_m\}_{m\ge0}\subset\mathbb Z[k]$ by
\[
N_0=0,\qquad N_1=1,\qquad N_2=k,\qquad
N_m=k\,N_{m-1}+N_{m-3}\quad(m\ge3).
\]
\end{definition}

The generating function and coefficient formula below will be used to compare these polynomials with the characteristic factors in the case $n=3$.

\begin{proposition}[Generating function {\cite[Eq.~(2.4)]{RamirezSirvent2015}}]\label{prop:kNarayana-gf}
One has
\[
\sum_{m\ge0}N_m\,z^m=\frac{z}{1-kz-z^3}.
\]
\end{proposition}

\begin{proposition}[Binomial expansion]\label{prop:kNarayana-binomial}
For every $m\ge1$,
\[
N_m=\sum_{j=0}^{\lfloor (m-1)/3\rfloor}\binom{m-1-2j}{j}\,k^{\,m-1-3j}.
\]
\end{proposition}

\begin{proof}
Expand Proposition~\ref{prop:kNarayana-gf} as
\[
\frac{z}{1-kz-z^3}
=\sum_{j\ge0}\frac{z^{3j+1}}{(1-kz)^{j+1}}
=\sum_{j\ge0}z^{3j+1}\sum_{q\ge0}\binom{q+j}{j}(kz)^q,
\]
using the binomial-series expansion, and read off the coefficient of $z^m$ (so $q=m-1-3j$).
\end{proof}

\subsection{A direct dictionary with $\Phi_{3,r}$}\label{app:Phi3-kN}

Here we use our convention $\Phi_{n,r}(x)=\det(xI-A_{n,r})$.

\begin{proposition}[A dictionary for $n=3$]\label{prop:N-Phi3}
For every $m\ge4$,
\[
N_m=(-1)^{m+1}\,\Phi_{3,m-3}(-k).
\]
\end{proposition}

\begin{proof}
Specializing Lemma~\ref{lem:Phi-binomial} to $n=3$ and $r=m-3$ gives
\[
\Phi_{3,m-3}(x)=\sum_{j=0}^{\lfloor (m-1)/3\rfloor}(-1)^j\binom{m-1-2j}{j}\,x^{\,m-1-3j}.
\]
Substituting $x=-k$ yields
\[
\Phi_{3,m-3}(-k)=\sum_{j=0}^{\lfloor (m-1)/3\rfloor}(-1)^{m-1-2j}\binom{m-1-2j}{j}\,k^{\,m-1-3j}.
\]
Multiplying by $(-1)^{m+1}$ removes the remaining sign, since $(-1)^{m+1}(-1)^{m-1-2j}=1$.
Therefore $N_m=(-1)^{m+1}\Phi_{3,m-3}(-k)$, as claimed.
\end{proof}

Proposition~\ref{prop:N-Phi3} identifies the $n=3$ characteristic factors with $k$-Narayana polynomials up to a simple sign twist,
so we can import standard root-information for $N_m(k)$.

\subsection{A root-property package for $k$-Narayana}\label{app:kNarayana-roots}

The next proposition bundles the three basic “root statistics” of $N_m(k)$ that we will use:
the multiplicity at $k=0$, simplicity of nonzero roots, and the complete list of rational nonzero roots.

\begin{proposition}[Root properties of $N_m$]\label{prop:kNarayana-root-package}
Let $m\ge 1$. Then:
\begin{enumerate}
\item[\textup{(i)}] The root $k=0$ has multiplicity
\[
\ord_{k} N_m=
\begin{cases}
0,& m\equiv 1\pmod 3,\\
1,& m\equiv 2\pmod 3,\\
2,& m\equiv 0\pmod 3.
\end{cases}
\]
\item[\textup{(ii)}] Every \emph{nonzero} root of $N_m$ is simple.
\item[\textup{(iii)}] The only possible rational nonzero root of $N_m$ is $k=-1$, and it occurs if and only if
\[
m\in\{4,13\}.
\]
\end{enumerate}
\end{proposition}

\begin{proof}
(i) By Proposition~\ref{prop:kNarayana-binomial} we have, for every $m\ge1$,
\[
N_m=\sum_{j=0}^{\lfloor (m-1)/3\rfloor}\binom{m-1-2j}{j}\,k^{\,m-1-3j}.
\]
The smallest exponent of $k$ among the terms is attained when $j=\lfloor (m-1)/3\rfloor$, hence
\[
\ord_k N_m = (m-1)-3\Big\lfloor \frac{m-1}{3}\Big\rfloor,
\]
which is exactly $0,1,2$ according as $m\equiv 1,2,0\pmod 3$.

\textup{(ii)} The cases $m=1,2,3$ are immediate. For $m\ge4$ we use the dictionary in Proposition~\ref{prop:N-Phi3}:
\[
N_m=(-1)^{m+1}\,\Phi_{3,m-3}(-k).
\]
Since the prefactor $(-1)^{m+1}$ is nonzero, a nonzero $k_0$ is a root of $N_m$ if and only if
$x=-k_0$ is a nonzero root of $\Phi_{3,m-3}(x)$, i.e., $-k_0$ is a nonzero eigenvalue of $A_{3,m-3}$.
By Corollary~\ref{cor:simple-nonzero} in the main text, every nonzero eigenvalue of $A_{3,m-3}$ is simple, hence
$\Phi_{3,m-3}$ has a simple root at $x=-k_0$. Therefore $N_m$ has a simple root at $k=k_0$.

\textup{(iii)}
By \textup{(i)}, after removing the maximal power of $k$ from $N_m$ the remaining polynomial has constant term $1$,
so any rational root must be an integer dividing $1$, hence belongs to $\{\pm 1\}$.
Moreover $N_m(1)>0$ for all $m$ (the recursion at $k=1$ produces a strictly positive integer sequence), so $k=1$ is never a root.
Finally,
\[
N_m(-1)=0
\iff
\Phi_{3,m-3}(1)=0
\iff
m-3\in\{1,10\}
\iff
m\in\{4,13\},
\]
where the middle equivalence is exactly the Padovan criterion (Theorem~\ref{thm:padovan-main}) specialized to $n=3$.
\end{proof}

\section{A worked example: the four-piece kit}\label{app:examples}

In this appendix we work out the example \(\Gamma_{3,8}\) explicitly, illustrating the layer partition, the induced \(3\)-cyclic block form, the block factors, and the cyclic core.
Within each layer, vertices are listed in increasing order.

\paragraph{The digraph $\Gamma_{3,8}$}
Let us work out explicitly the example $\Gamma_{3,8}$ shown in Figure~\ref{fig:spectrum_Gamma38}.
The left panel depicts the digraph and the right panel shows the corresponding nonzero spectrum.

For $(n,r)=(3,8)$ we have $|V(\Gamma_{3,8})|=rn-2r+2=10$, hence $V(\Gamma_{3,8})=\{1,2,\dots,10\}$.

By the construction in Section~\ref{subsec:Gamma}, the directed edge set is
\[
\begin{aligned}
E(\Gamma_{3,8})=\{&
(1,2),(2,3),(3,1),\ (1,4),(4,3),\
(3,5),(5,4),\ (4,6),(6,5),\\ &(5,7),(7,6),\
(6,8),(8,7),\ (7,9),(9,8),\ (8,10),(10,9)\}.
\end{aligned}
\]

The grading records the layer of each vertex in the period-\(3\) structure, so that every edge advances the layer index by one modulo \(3\).

\paragraph{A grading and layers}
Define $\ell:V(\Gamma_{3,8})\to \mathbb Z/3\mathbb Z$ by
\[
\ell(1)=0,\ \ell(2)=1,\ \ell(3)=2,\ \ell(4)=1,\ \ell(5)=0,\ \ell(6)=2,\ 
\ell(7)=1,\ \ell(8)=0,\ \ell(9)=2,\ \ell(10)=1.
\]
Then the layers $V_{3,8}^{(c)}:=\{u\in V(\Gamma_{3,8}) : \ell(u)\equiv c \pmod 3\}$ are
\[
V_{3,8}^{(0)}=\{1,5,8\},\qquad
V_{3,8}^{(1)}=\{2,4,7,10\},\qquad
V_{3,8}^{(2)}=\{3,6,9\}.
\]

\paragraph{Permutation to a $3$-cyclic block form}
Let $S_{3,8}$ be the permutation matrix that reorders the basis according to
\[
\sigma_{3,8}:(1,2,3,4,5,6,7,8,9,10)\mapsto(1,5,8\,|\,2,4,7,10\,|\,3,6,9),
\]
that is, in the layer order $(V^{(0)}_{3,8}\,|\,V^{(1)}_{3,8}\,|\,V^{(2)}_{3,8})$. Put
$B_{3,8}:=S_{3,8}A_{3,8}S_{3,8}^{-1}$. Then $B_{3,8}$ has the $3$-cyclic block form
\[
B_{3,8}=
\begin{bmatrix}
0 & B^{(0)}_{3,8} & 0\\
0 & 0 & B^{(1)}_{3,8}\\
B^{(2)}_{3,8} & 0 & 0
\end{bmatrix},
\qquad (|V^{(0)}_{3,8}|,|V^{(1)}_{3,8}|,|V^{(2)}_{3,8}|)=(3,4,3),
\]
where $B^{(c)}_{3,8}$ records edges from $V^{(c)}_{3,8}$ to $V^{(c+1)}_{3,8}$ (indices modulo $3$). Explicitly,
\[
B^{(0)}_{3,8}=
\begin{bmatrix}
1&1&0&0\\
0&1&1&0\\
0&0&1&1
\end{bmatrix},\qquad
B^{(1)}_{3,8}=
\begin{bmatrix}
1&0&0\\
1&1&0\\
0&1&1\\
0&0&1
\end{bmatrix},\qquad
B^{(2)}_{3,8}=
\begin{bmatrix}
1&1&0\\
0&1&1\\
0&0&1
\end{bmatrix}.
\]

\paragraph{Cyclic core and the polygonal spectrum}
Define the core on the layer $V^{(0)}_{3,8}$ by
\[
K_{3,8}:=B^{(0)}_{3,8}B^{(1)}_{3,8}B^{(2)}_{3,8}
=
\begin{bmatrix}
2&3&1\\
1&3&3\\
0&1&3
\end{bmatrix}.
\]
Its characteristic polynomial is
\[
\det(xI-K_{3,8})=x^{3}-8x^{2}+15x-4,
\]
so $K_{3,8}$ has three real, positive, simple eigenvalues $0<\mu_1<\mu_2<\mu_3$.
Set
\[
\theta:=\frac13\arccos\!\left(\frac{26\sqrt{19}}{361}\right).
\]
Then (casus irreducibilis) one may write
\[
\mu_1=\frac{8}{3}+\frac{2\sqrt{19}}{3}\cos\!\left(\theta+\frac{2\pi}{3}\right),\qquad
\mu_2=\frac{8}{3}+\frac{2\sqrt{19}}{3}\cos\!\left(\theta-\frac{2\pi}{3}\right),\qquad
\mu_3=\frac{8}{3}+\frac{2\sqrt{19}}{3}\cos\theta.
\]
Numerically,
\[
\mu_1\approx 0.3186693564,\qquad
\mu_2\approx 2.3579263675,\qquad
\mu_3\approx 5.3234042761.
\]
Let $\omega_3:=\exp(2\pi i/3)$. By Proposition~\ref{prop:charpoly-reduction},
\[
\Spec(A_{3,8})
=\{0\}\ \cup\ \bigcup_{j=1}^3\{\mu_j^{1/3},\ \mu_j^{1/3}\omega_3,\ \mu_j^{1/3}\omega_3^2\}.
\]
Thus the nonzero spectrum of $A_{3,8}$ is the disjoint union of three regular triangles,
each with one vertex on the positive real axis; moreover $0$ has algebraic multiplicity $1$. See Figure~\ref{fig:spectrum_Gamma38} for the corresponding plot of these three triangle packets and the circle of radius $(27/4)^{1/3}$.

\end{document}